\documentclass[11]{amsart}

\usepackage{amsmath, amstext, amsgen, amsbsy, amsopn, amsfonts, amssymb, graphicx,psfrag}
\usepackage{amsthm}
\usepackage{comment}
\usepackage[normalem]{ulem}
\usepackage{pdfsync}
\usepackage{epstopdf}
\usepackage{color}
\usepackage[thinlines]{easytable}
\usepackage{multirow}
\usepackage{multicol}
\usepackage{hyperref}

\def\Z{\mathbb{Z}}

\newcommand{\semi}{{\rtimes}}

\def\TSG{{\mathrm{TSG}}}
\def\fix{{\mathrm{fix}}}
\def\Aut{{\mathrm{Aut}}}

\newtheorem{theorem}{Theorem}[section]
\newtheorem{lemma}[theorem]{Lemma}

\newtheorem{question}{Question}
\theoremstyle{definition}
\newtheorem*{definition}{Definition}
\newtheorem*{subgroupcor}{Subgroup Theorem}
\newtheorem*{smith}{Smith Theory}
\newtheorem*{finiteorder}{Finite Order Theorem}
\newtheorem*{subgraph}{Subgraph Lemma}

\theoremstyle{remark}

\begin{document}

\title{Topological Symmetry Groups of the Petersen graphs}

\author{Deion Elzie, Samir Fridhi, Blake Mellor, Daniel Silva, and Robin Wilson}
\address{Loyola Marymount University, 1 LMU Drive, Los Angeles, CA 90045}
%\ead{blake.mellor@lmu.edu}
\email{blake.mellor@lmu.edu}
\address{Loyola Marymount University, 1 LMU Drive, Los Angeles, CA 90045}
\email{robin.wilson@lmu.edu}
\begin{abstract} 
The {\em topological symmetry group} of an embedding $\Gamma$ of an abstract graph $\gamma$ in $S^3$ is the group of automorphisms of $\gamma$ which can be realized by homeomorphisms of the pair $(S^3, \Gamma)$.  These groups are motivated by questions about the symmetries of molecules in space.  The Petersen family of graphs is an important family of graphs for many problems in low dimensional topology, so it is desirable to understand the possible groups of symmetries of their embeddings in space.  In this paper, we find all the groups which can be realized as topological symmetry groups for each of the graphs in the Petersen Family. Along the way, we also complete the classification of the realizable topological symmetry groups for $K_{3,3}$.
\end{abstract}

\date{}
\maketitle

\section{Introduction}\label{S:intro}

In molecular chemistry, the properties of certain molecules are often strongly influenced by their symmetries in space. Molecules with the same chemical structure but whose embeddings in space are not equivalent are called {\bf stereoisomers}, and can often have quite different effects. Historically, chemists have been most interested in rigid symmetries of molecules, induced by rotations and reflections; but as our ability to synthesize molecules advances there is increased interest in symmetries of more ``flexible'' molecules such as DNA and various long polymers.  In these cases, there may be homeomorphisms of space that deform the molecule and map it back to itself, but in a way that cannot be realized by a combination of rotations and reflections. To describe the symmetries of these more complex molecules, Jon Simon \cite{si} introduced the {\bf topological symmetry group} of an embedded graph as the group of automorphisms of the graph induced by homeomorphisms of $S^3$. The topological approach to molecular chemistry has been fruitful, and is an active area for current and future research \cite{f3, ro}.

Since Simon's original paper, considerable work has been done on the topic of topological symmetry groups.  This is a generalization of the study of symmetries and achirality of knots and links, so it is a natural and important problem in low dimensional topology. In many cases, the motivating question is: given an abstract graph (e.g. chemical structure of a molecule), what are the possible topological symmetry groups over all embeddings of the graph in $S^3$? This problem has been solved for several important families of graphs, including complete graphs \cite{cf, cfo, fmn3, fnt},  complete bipartite graphs \cite{hmp} and M\"{o}bius ladders \cite{f1, fl}, as well as for some individually interesting graphs \cite{cfhltv, flw}.  There are also some broad results restricting possible topological symmetry groups for any 3-connected graph \cite{fnpt}.

The Petersen family of graphs, shown in Figure \ref{F:petersen} are an important family of graphs in low-dimensional topology.  Most importantly, they form the complete set of minor-minimal intrinsically linked graphs \cite{rst}. As a result, they are important examples to explore many other issues surrounding linked and knotted cycles in spatial graphs \cite{ni, od}. They include such well-known graphs as $K_6$, the complete graph on 6 vertices, and the Petersen graph itself (which gives the family its name).  The family consists of all graphs related to the Petersen graph by some sequence of $\nabla Y$ or $Y\nabla$ moves (``triangle-Y'' or ``Y-triangle''). A $\nabla Y$ move is an operation in which a triangle in the graph is replaced by a degree-three vertex in the shape of a ``Y", while a $Y\nabla$ move is the reverse, as shown in Figure \ref{F:triangleY}.

\begin{figure} [htbp]
$$\scalebox{.8}{\includegraphics{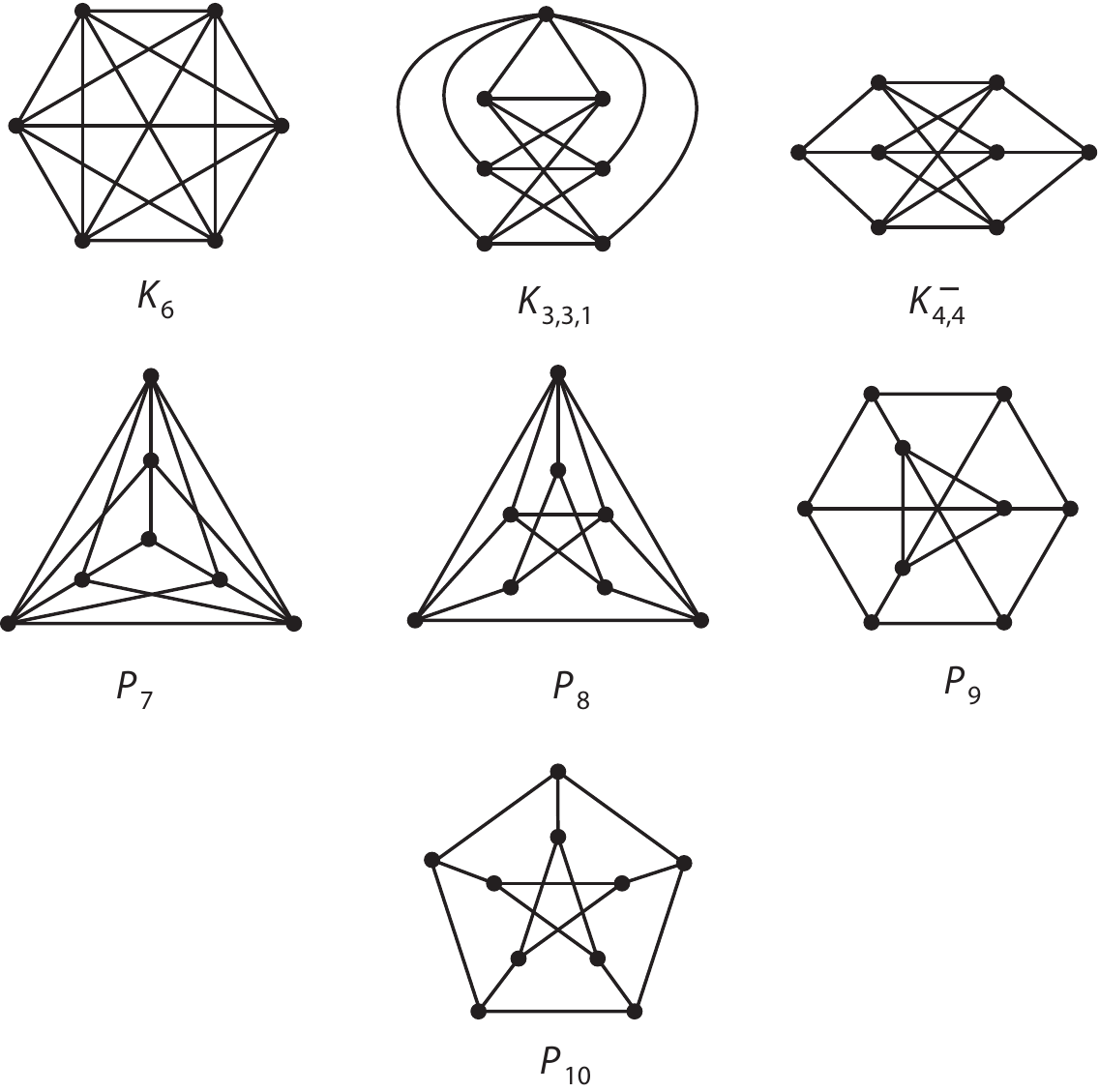}}$$
\caption{The Petersen family of graphs.}
\label{F:petersen}
\end{figure}

\begin{figure} [htbp]
$$\scalebox{.6}{\includegraphics{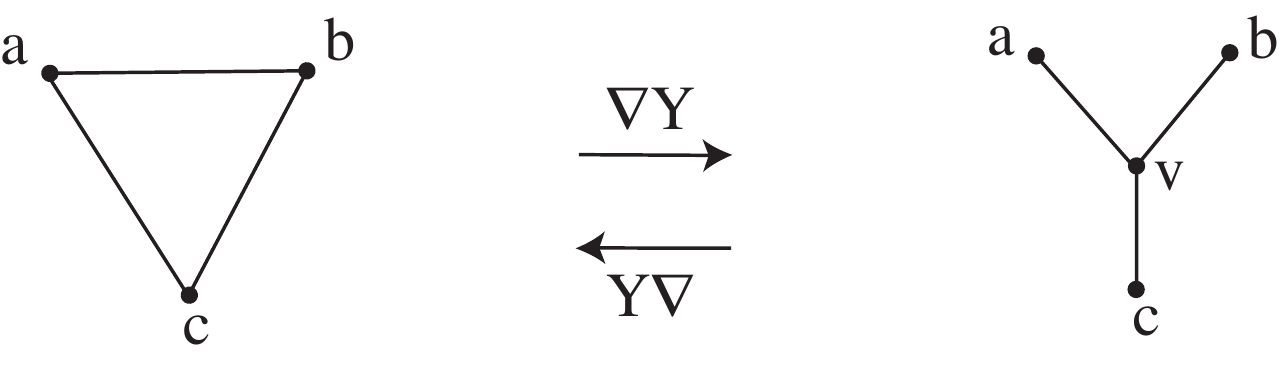}}$$
\caption{$\nabla Y$ and $Y\nabla$ moves.}
\label{F:triangleY}
\end{figure}

Our main purpose in this paper is to determine what groups can occur as topological symmetry groups for some embedding of a graph in the Petersen family.  The groups that can arise as  topological symmetry groups of some embedding of the Petersen graph itself were determined by Chambers, et. al. in \cite{cfhltv}.  Here we build on this result and complete the classification of topological symmetry groups for the remaining graphs in the Petersen family.  Aside from the Petersen graph, which we denote by $P_{10}$, and the complete graph $K_6$, the other graphs in the family are the complete tripartite graph $K_{3,3,1}$, the graph $K_{4,4}$ with an edge removed (which we denote $K_{4,4}^-$), and the graphs we denote $P_7$, $P_8$, $P_9$.

\section{Background and definitions}

Before we begin, we introduce some terminology, and some important tools from other papers.  An {\em abstract graph} (or just a {\em graph}) $\gamma$ is a pair $(V, E)$, where $V$ is a set of vertices and $E \subseteq V \times V$ is a set of edges. Since our graphs are not directed, we will denote the edge between vertices $v$ and $w$ by the unordered pair $\{v,w\}$. An embedding $f: \gamma \rightarrow S^3$ means (1) an embedding of the vertices $V$ in $S^3$ and (2) for each edge $\{v, w\}$, an embedding $f_{v,w}: [0,1] \rightarrow S^3$ such that $f_{v,w}(0) = f(v)$ and $f_{v,w}(1) = f(w)$, and the embeddings of distinct edges intersect only at the endpoints. The image $\Gamma = f(\gamma)$ is called an {\em embedded graph} or a {\em spatial graph}.

An {\em automorphism} of a graph $\gamma$ is a bijection $\alpha: V \rightarrow V$ such that $\{v, w\} \in E$ if and only if $\{\alpha(v), \alpha(w)\} \in E$. The automorphisms of a graph $\gamma$ form a group (the {\em automorphism group}), denoted $\Aut(\gamma)$. To describe automorphism groups (and their subgroups), we will use the following standard terms:
\begin{itemize}
\item The {\em symmetry group} $S_n$ is the group of permutations of a set of $n$ objects.
\item The {\em dihedral group} $D_n$ is the group of size $2n$ with presentation $\langle m, r \mid m^2 = r^n = 1, rm = mr^{-1}\rangle$.
\item The {\em cyclic group} $\Z_n$ is the group of size $n$ with presentation $\langle r \mid r^n = 1\rangle$.
\item The {\em direct product} $G \times H$ is the Cartesian product $G \times H$ with the group operation $(g_1, h_1)(g_2, h_2) = (g_1g_2, h_1h_2)$.
\item The {\em semidirect product} $G \semi H$ is the Cartesian product $G\times H$ with the group operation $(g_1, h_1)(g_2, h_2) = (g_1\phi_{h_1}(g_2), h_1h_2)$, for some homomorphism $\phi: H \rightarrow \Aut(G)$. In our arguments, unless otherwise specified, $H = \Z_2 = \langle r \mid r^2 = 1\rangle$ and $\phi_r(g) = g^{-1}$ for every $g \in G$.
\end{itemize}

A homeomorphism of $S^3$ which takes an embedded graph $\Gamma$ to itself induces an automorphism on the underlying abstract graph. We are interested in which automorphisms can be induced in this way. In what follows, we will refer to a homeomorphism of $S^3$ taking an embedded graph $\Gamma$ to itself, as a homeomorphism of the pair $(S^3,\Gamma)$.

\begin{definition}\label{def1.2} 
Let $\gamma$ be an abstract graph, and let $\Gamma$ be an embedding of $\gamma$ in $S^3$.  We define the {\bf topological symmetry group} $\TSG(\Gamma)$ as the subgroup of $\Aut(\gamma)$ induced by homeomorphisms of $(S^3,\Gamma)$.  We define the {\bf orientation preserving topological symmetry group} $\TSG_+(\Gamma)$ as the subgroup of $\Aut(\gamma)$ induced by orientation preserving homeomorphisms of $(S^3,\Gamma)$.  
\end{definition}

\begin{definition}\label{def1.4} Let  $G$ be a group and let $\gamma$ denote an abstract graph. If there is some embedding $\Gamma$ of $\gamma$ in $S^3$ such that $\TSG(\Gamma)=G$ (resp. $\TSG_+(\Gamma)=G$),  then we say that the group $G$ is {\bf realizable}  (resp. {\bf positively realizable}) for $\gamma$. We will also say that a particular automorphism $\sigma \in \Aut(\gamma)$ is {\bf realizable} (resp. {\bf positively realizable}) if there is a homeomorphism (resp. orientation-preserving homeomorphism) of $(S^3, \Gamma)$ which induces $\sigma$ (for some embedding $\Gamma$).
\end{definition}

Note that for any embedding $\Gamma$ of an abstract graph, $\TSG_+(\Gamma)$ is an index two subgroup of $\TSG(\Gamma)$.  Further, if a group $G$ is positively realizable by some embedding $\Gamma$, then $G$ is also realizable \cite{cf}. This is because we can add identical chiral knots to every edge of the embedding $\Gamma$ to rule out any orientation-reversing homeomorphisms.  This yields a new embedding $\Gamma'$ with $\TSG(\Gamma') = \TSG_+(\Gamma') = \TSG_+(\Gamma) = G$.  As a result, for any graph $\gamma$, we are interested in finding (1) the groups which are positively realizable for $\gamma$ and (2) the groups which are realizable, but not positively realizable, for $\gamma$.

For example, these groups have been found for $K_6$ and the Petersen graph $P_{10}$:

\begin{theorem} \cite{cf}
The groups that can be {\bf positively} realized for the complete graph $K_6$ are:
$$D_6, D_5, D_3, D_2, \Z_6, \Z_5, \Z_3, \Z_2, D_3 \times D_3, D_3 \times \Z_3, \Z_3 \times \Z_3, (\Z_3 \times \Z_3) \semi \Z_2.$$
The groups that can be realized (but not positively realized) for $K_6$ are:
$$D_4, \Z_4, (D_3 \times D_3) \semi \Z_2, (\Z_3 \times \Z_3) \semi \Z_4.$$
\end{theorem}

\begin{theorem} \cite{cfhltv}
The groups that can be {\bf positively} realized for the Petersen graph $P_{10}$ are:
$$D_5, D_3, \Z_5, \Z_3, \Z_2.$$
The groups that can be realized (but not positively realized) for $P_{10}$ are:
$$\Z_5 \semi \Z_4 \text{ and } \Z_4.$$
\end{theorem}

We now turn to some of the important tools we will use in this paper. The following result is very useful in showing that certain groups can be realized as a topological symmetry group. (In \cite{fmn1}, this is called the Subgroup Corollary.)

\begin{subgroupcor}\cite{fmn1} 
\label{subgroup}
Let $\Gamma$ be an embedding of a 3-connected graph in $S^3$.  Suppose that $\Gamma$ contains an edge $e$ which is not pointwise fixed by any non-trivial element of $\TSG_+(\Gamma)$.  Then for every $H \leq \TSG_+(\Gamma)$, there is an embedding $\Gamma'$ of $\Gamma$ with $H = \TSG_+(\Gamma')$.
\end{subgroupcor}

The next three results give us powerful tools for restricting which automorphisms or groups can be realized for a particular graph $\gamma$.

\begin{finiteorder}\cite{f2}
Let $\phi$ be a non-trivial automorphism of a 3-connected graph $\gamma$ which is induced by a homeomorphism $h$ of $ (S^3, \Gamma)$ for some embedding $\Gamma$ of $\gamma$. Then there exists another embedding $\Gamma'$ of $\gamma$ such that the automorphism $\phi$ is induced by a finite order homeomorphism $f \in (S^3, \Gamma')$, and $f$ is orientation reversing if and only if $h \in (S^3, \Gamma)$ is orientation reversing. 
\end{finiteorder}

\begin{smith} \cite{sm}
Let $h$ be a non-trivial finite order homeomorphism of $S^3$.  If $h$ is orientation preserving, then $\fix(h)$ is either the empty set or is homeomorphic to $S^1$.  If $h$ is orientation reversing, then $\fix(h)$ is homeomorphic to either $S^0$ (two distinct points) or $S^2$.
\end{smith}

\begin{subgraph}
Let $\gamma$ be an abstract graph and $\gamma'$ be a subgraph of $\gamma$ so that every automorphism of $\gamma$ fixes $\gamma'$ setwise.  Assume that $\Gamma$ is an embedding of $\gamma$ and let $\Gamma'$ be the embedding of $\gamma'$ induced by $\Gamma$.  Then $\TSG(\Gamma) \leq \TSG(\Gamma')$ and $\TSG_+(\Gamma) \leq \TSG_+(\Gamma')$.  
\end{subgraph}
\begin{proof}
Observe that every element in $\TSG(\Gamma)$ is a homeomorphism of $S^3$ taking $\Gamma$ to itself setwise.  Hence, every element in $\TSG(\Gamma)$ also takes $\Gamma'$ to itself setwise and so is an element of $\TSG(\Gamma')$. Thus, $\TSG(\Gamma)$ is a subset of $\TSG(\Gamma')$ that is also a group under the same operation.  The same argument holds for $\TSG_+(\Gamma)$. 
\end{proof}

In the remainder of the paper, we will classify the realizable topological symmetry groups for the remaining graphs in the Petersen family: $K_{3,3,1}$, $K_{4,4}^-$, $P_7$, $P_8$ and $P_9$.  First, however, we will complete the classification of the topological symmetry groups for $K_{3,3}$ begun by Flapan and Lawrence \cite{fl}. Since $K_{3,3}$ is a subgraph of several of the graphs in the Petersen family, this will help in our classification of the other graphs, particularly $K_{3,3,1}$ and $K_{4,4}^-$.

%%%%%%%%%%%%%%%%%%%%%%%%%%%%%%%%%%
\section{The graph $K_{3,3}$}\label{S:K33}

The graph $K_{3,3}$ is the complete bipartite graph with two sets of three vertices, as shown in Figure \ref{F:K33}. The positively realizable groups were determined by Flapan and Lawrence \cite{fl}, so our primary goals are to review their results, to give a more detailed accounting of the subgroups of $\Aut(K_{3,3})$, and to determine the groups which are realizable but not positively realizable for $K_{3,3}$.

\begin{figure} [htbp]
$$\scalebox{.7}{\includegraphics{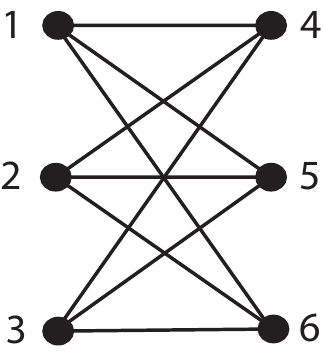}}$$
\caption{The complete bipartite graph $K_{3,3}$.}
\label{F:K33}
\end{figure}

We begin by describing the automorphism group for $K_{3,3}$. Using the labeling in Figure \ref{F:K33}, we can permute the vertices in $\{1, 2, 3\}$ and we can (independently) permute the vertices in $\{4,5,6\}$.  We can also interchange the two sets of vertices.  It is convenient to view these automorphisms as a set of permutations of the six vertices, hence as a subset of $S_6$.  As a result, we have (also see \cite{fl}):
$$\Aut(K_{3,3}) \cong (S_3 \times S_3) \semi \Z_2 = \langle (12), (123), (45), (456), (14)(25)(36) \rangle \subset S_6$$
(This group is sometimes referred to as the {\em wreath product} of $S_3$ by $\Z_2$. Here the semidirect product uses the homomorphism $\phi: \Z_2 \rightarrow \Aut(S_3 \times S_3)$ where $\phi_r(a,b) = (b,a)$ for each $(a,b) \in S_3 \times S_3$.)

\subsection{Subgroups of $\Aut(K_{3,3})$.}\label{SS:K33subgps}  We wish to list (up to isomorphism) the subgroups of $\Aut(K_{3,3})$.  Nikkuni and Taniyama \cite{nt} listed representatives from each (non-trivial) conjugacy class of $\Aut(K_{3,3})$ and determined which classes were realizable by orientation preserving automorphisms of $S^3$ (positively realizable), and which were realizable by orientation reversing automorphisms (negatively realizable).  No automorphisms are in both categories.
\begin{itemize}
	\item[] {\em Positively Realizable:} (12)(45), (14)(25)(36), (123), (123)(456), (142536)
	\item[] {\em Negatively Realizable:} (12), (1425)(36), (12)(456)
\end{itemize}

The following table lists the number of permutations in each conjugacy class:

\begin{center}
\begin{tabular}{|c|c|}
\hline
Conjugacy class & Number of elements \\ \hline
identity & 1 \\
(12) & 6 \\
(12)(45) & 9 \\
(14)(25)(36) & 6 \\
(123) & 4 \\
(123)(456) & 4 \\
(1425)(36) & 18 \\
(142536) & 12 \\
(12)(456) & 12 \\ \hline
Total & 72 \\ \hline
\end{tabular}
\end{center}

To find the subgroups of $\Aut(K_{3,3})$, we begin with the subgroups of $S_6$. Up to isomorphism, there are 29 distinct subgroups (see, for example, \cite{fl}; these can also be generated from a computational group theory program such as GAP).  Recall that $S_3 \cong D_3$; we will use $D_3$ in this list.  We list the groups in order of decreasing size.
$$S_6, A_6, S_5, (D_3 \times D_3) \semi \Z_2, A_5, S_4 \times \Z_2, D_3 \times D_3, (\Z_3 \times \Z_3) \semi \Z_4, S_4, A_4 \times \Z_2,$$
$$\Z_5 \semi \Z_4, (\Z_3\times \Z_3) \semi \Z_2, D_3 \times \Z_3, D_4 \times \Z_2, A_4, D_6, D_5, \Z_3 \times \Z_3, \Z_2 \times \Z_4, \Z_2 \times \Z_2 \times \Z_2,$$
$$D_4, D_3, \Z_6, \Z_5, D_2, \Z_4, \Z_3, \Z_2, id$$
We can immediately eliminate the groups whose order does not divide 72: $S_6$, $A_6$, $S_5$, $A_5$, $S_4 \times \Z_2$, $\Z_5 \semi \Z_4$, $D_4 \times \Z_2$, $D_5$, $\Z_5$.  Next, observe that $\Aut(K_{3,3}) = (D_3 \times D_3) \semi \Z_2$ has 8 elements of order 3, all of which commute.  $S_4$ and $A_4$ also have 8 elements of order 3, but in these groups they do {\em not} all commute.  Hence $\Aut(K_{3,3})$ cannot contain $S_4$, $A_4$, or $A_4 \times \Z_2$ as a subgroup.

We can also observe that the only elements of $\Aut(K_{3,3})$ of order 4 are conjugate to $(1425)(36)$, and it is easy to check that this permutation does not commute with any order two permutation (other than its square $(12)(45)$).  Hence $\Z_2 \times \Z_4$ is not a subgroup of $\Aut(K_{3,3})$.

Finally, we consider the group $\Z_2 \times \Z_2 \times \Z_2$.  This group contains 7 involutions, all of which commute with each other.  There are three conjugacy classes of involutions in $\Aut(K_{3,3})$: $(14)(25)(36)$ does not commute with any other involutions; $(12)(45)$ commutes with only two other involutions (namely $(12)$ and $(45)$); and $(12)$ commutes with six other involutions (namely $(45)$, $(46)$, $(56)$, $(12)(45)$, $(12)(46)$ and $(12)(56)$).  However, the six involutions which commute with $(12)$ do not commute with each other.  So it is impossible to find seven involutions in $\Aut(K_{3,3})$, all of which commute with each other.  Hence $\Z_2 \times \Z_2 \times \Z_2$ is not a subgroup of $\Aut(K_{3,3})$.

The remaining subgroups of $S_6$ can all be realized as subgroups of $\Aut(K_{3,3})$.  Table \ref{Ta:K33} gives a generating set of permutations for a subgroup in each isomorphism class.

\begin{table}[htbp]
\begin{center}
\begin{tabular}{|l|l|}
\hline
Subgroup & Generating set \\ \hline
$(D_3 \times D_3) \semi \Z_2$ & $ (12), (123), (45), (456), (14)(25)(36) $ \\
$(\Z_3 \times \Z_3) \semi \Z_4$ & $ (123), (456), (1425)(36) $\\
$(\Z_3\times \Z_3) \semi \Z_2$ & $ (123), (456), (14)(25)(36) $\\
$D_3 \times D_3$ & $ (12), (123), (45), (456) $\\
$D_3 \times \Z_3$ & $ (12), (123), (456) $ \\
$\Z_3 \times \Z_3$ & $ (123), (456) $ \\
$D_6$ & $ (12)(56), (142536)$ \\
$D_4$ & $ (12), (1425)(36)$ \\
$D_3$ & $ (12), (123) $ \\
$D_2$ & $ (12), (45) $ \\
$\Z_6$ & $ (142536) $ \\
$\Z_4$ & $ (1425)(36) $ \\
$\Z_3$ & $ (123) $ \\
$\Z_2$ & $ (12) $ \\ \hline
\end{tabular}
\end{center}
\caption{Generating sets for subgroups of $\Aut(K_{3,3})$.}
\label{Ta:K33}
\end{table}

\begin{theorem}\label{T:K33subgroups}
The subgroups of $\Aut(K_{3,3}) = (D_3 \times D_3) \semi \Z_2$, up to isomorphism, are
$$(D_3 \times D_3) \semi \Z_2, D_3 \times D_3, (\Z_3 \times \Z_3) \semi \Z_4, (\Z_3\times \Z_3) \semi \Z_2, D_3 \times \Z_3,$$
$$D_6, \Z_3 \times \Z_3, D_4, D_3, \Z_6, D_2, \Z_4, \Z_3, \Z_2, id$$
\end{theorem}

\subsection{Topological symmetry groups of $K_{3,3}$.}

Flapan and Lawrence determined which subgroups of $\Aut(K_{3,3})$ are positively realizable.

\begin{theorem}[\cite{fl}] \label{T:K33TSG+}
The non-trivial groups which occur as $\TSG_+(\Gamma)$ for some embedding $\Gamma$ of $K_{3,3}$ are:
$$D_3 \times D_3, (\Z_3\times \Z_3) \semi \Z_2, D_3 \times \Z_3, D_6, \Z_3 \times \Z_3, D_3, \Z_6, D_2, \Z_3, \Z_2$$
\end{theorem}

We will show that the remaining four non-trivial subgroups are realizable (although not positively realizable). 

\begin{theorem} \label{T:K33TSG}
The groups which are realizable but {\bf not} positively realizable for $K_{3,3}$ are: $(D_3 \times D_3) \semi \Z_2$, $(\Z_3 \times \Z_3) \semi \Z_4$, $D_4$ and $\Z_4$.
\end{theorem}

\begin{proof}

\begin{figure} [htbp]
$$\scalebox{.8}{\includegraphics{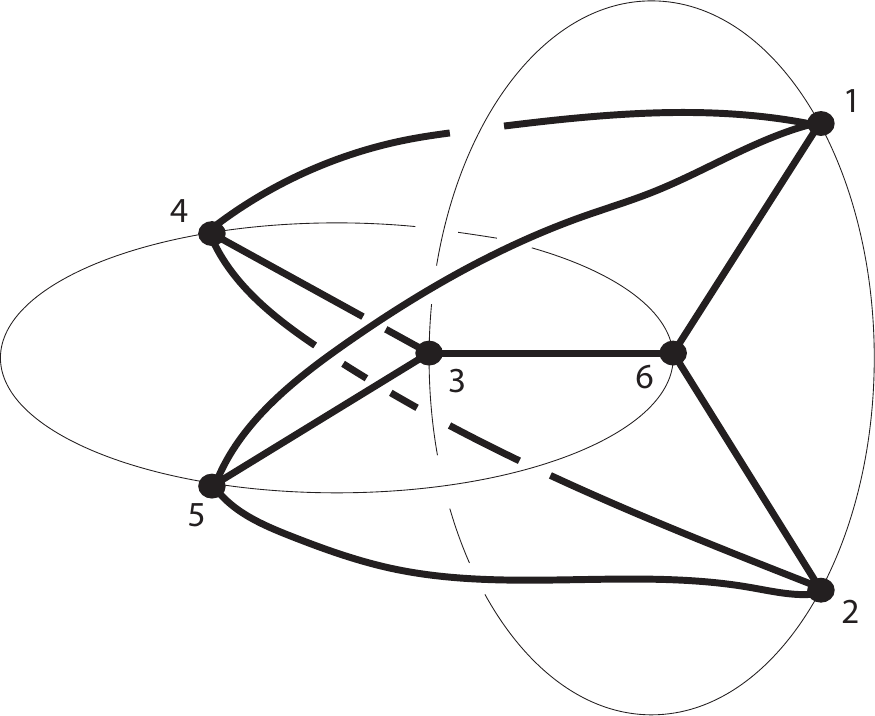}}$$
\caption{An embedding $\Gamma$ of the graph $K_{3,3}$ with $\TSG(\Gamma) = (D_3 \times D_3) \semi \Z_2$. The geodesic circles are shown for reference, but are not part of the graph.}
\label{F:K33embedding1}
\end{figure}

Figure \ref{F:K33embedding1} illustrates an embedding of $K_{3,3}$ with topological symmetry group $(D_3 \times D_3) \semi \Z_2$. The vertices $\{1, 2, 3\}$ are arranged symmetrically around a geodesic circle, while the vertices $\{4, 5, 6\}$ are arranged around a complementary geodesic circle; in other words, the two circles are the cores of complementary tori whose union is $S^3$. The permutation $(12)$ is realized by the reflection in the sphere containing $\{4, 5, 6\}$ and perpendicular to the circle through $\{1, 2, 3\}$; similarly, $(45)$ is realized by reflection in the sphere containing $\{1,2, 3\}$ and perpendicular to the circle through $\{4, 5, 6\}$ (in Figure \ref{F:K33embedding1}, these two spheres can be pictured as the planes containing the two geodesic circles).  The permutations $(123)$ and $(456)$ are realized by the rotations of order 3 around each of the geodesic circles.  Finally, the permutation $(14)(25)(36)$ is realized by the homeomorphism of $S^3$ that interchanges the two complementary tori. Together, these permutations generate $\Aut(K_{3,3}) = (D_3 \times D_3) \semi \Z_2$.

\begin{figure} [htbp]
$$\scalebox{1.6}{\includegraphics{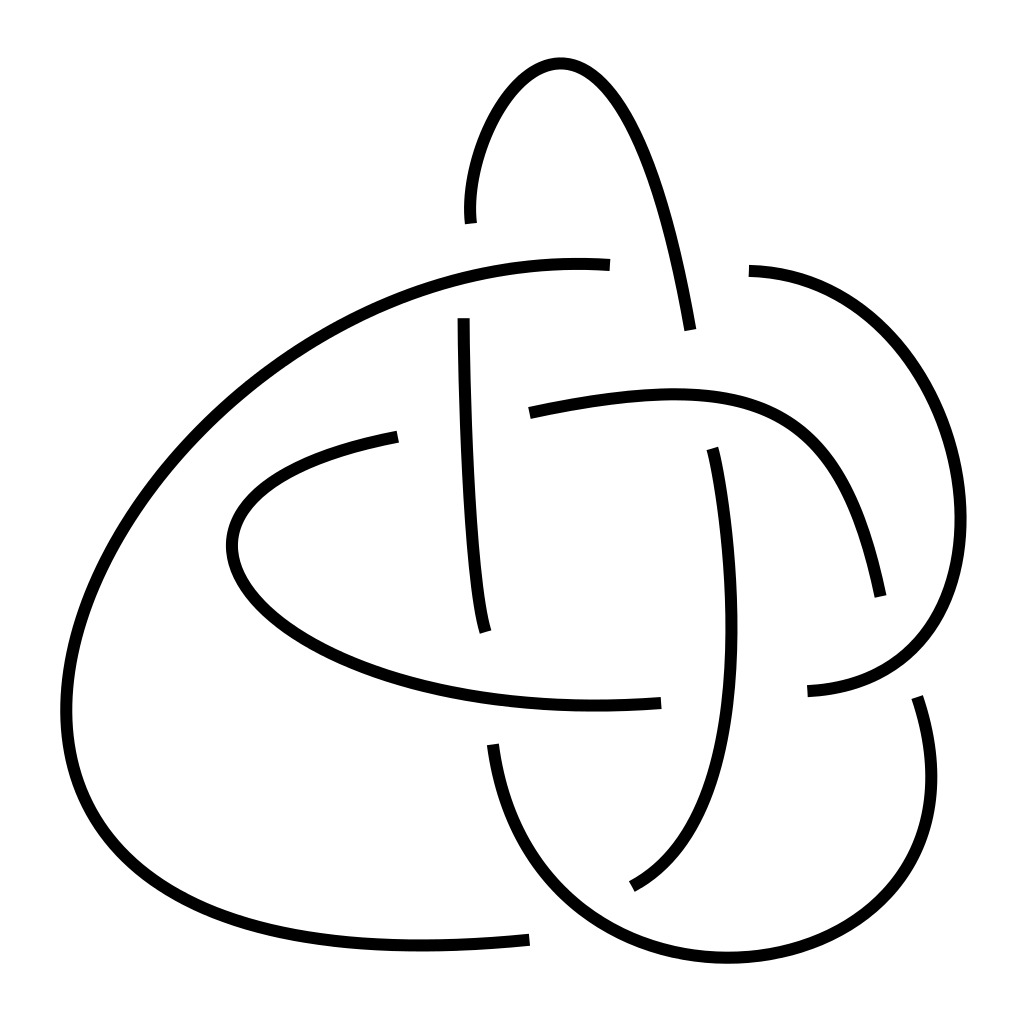}} \qquad \scalebox{.4}{\includegraphics{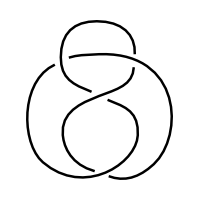}}$$
\caption{The negative amphicheiral knot $8_{17}$  and the fully amphicheiral figure 8 knot. }
\label{F:knots}
\end{figure}

To obtain an embedding with topological symmetry group $(\Z_3 \times \Z_3) \semi \Z_4$, we add the knot $8_{17}$ (shown in Figure \ref{F:knots}) to each edge of the embedding in Figure \ref{F:K33embedding1}.  This is the knot with fewest crossings which is {\em negative amphicheiral}, meaning that it is not equivalent to its inverse (the result of changing orientation) or to its mirror image, but it {\em is} equivalent to the inverse of its mirror image. The permutation $(12)$ must be realized by an orientation-reversing homeomorphism which fixes the edge $\{3,6\}$; this would require an isotopy between $8_{17}$ and its mirror image, which is impossible.  So this embedding does not allow the transposition $(12)$.  However, it does allow the rotations $(123)$ and $(456)$ as before.  Moreover, it allows the permutation $(1425)(36) = (12)(14)(25)(36)$, generated by first interchanging the complementary tori, and then reflecting in a sphere. This combination first inverts the edges (and their respective knots), and then takes the mirror images.  Since $8_{17}$ is equivalent to the inverse of its mirror image, the result is isotopic to the original embedding.  So the topological symmetry group contains $(\Z_3 \times \Z_3) \semi \Z_4$ (generated by the 3-cycles and the 4-cycle); since this group has index 2 in $\Aut(K_{3,3})$, and the topological symmetry group does not contain $(12)$ (and so is not equal to $\Aut(K_{3,3})$), the topological symmetry group must equal $(\Z_3 \times \Z_3) \semi \Z_4$.

To realize $D_4$ as the topological symmetry group, we begin with the embedding in Figure \ref{F:K33embedding1} and add the fully amphicheiral figure 8 knot (see Figure \ref{F:knots}) to the edge $\{3,6\}$. The figure 8 knot is isotopic to both its inverse and its mirror image, so this embedding is preserved by the reflection $(12)$ (which sends edge $\{3,6\}$. to its mirror image), and by the homeomorphism realizing $(14)(25)(36)$ (which sends edge $\{3,6\}$. to its inverse). These two permutations generate a subgroup of $\Aut(K_{3,3})$ isomorphic to $D_4$ (note that the 4-cycle $(1425)(36)$ is the product $(12)(14)(25)(36)$); however, since this embedding does not allow the permutation $(123)$ (since only edge $\{3,6\}$ has a knot tied in it), its topological symmetry group is not the full automorphism group.  But the only proper subgroup of $\Aut(K_{3,3})$ that contains $D_4$ is $D_4$ itself, so this embedding has a topological symmetry group isomorphic to $D_4$.

Finally, to realize $\Z_4$, we add the figure 8 knot to the edge $\{3,6\}$., as in the last paragraph, and then add the knot $8_{17}$ to the edges $\{1,4\}$, $\{1,5\}$, $\{2,4\}$ and $\{2,5\}$.  As before, the knot $8_{17}$ does not allow the permutations $(12)$ (since it is chiral) or $(14)(25)(36)$ (since it is noninvertible), but it does allow the product $(1425)(36)$ (since it is isotopic to the inverse of its mirror image). This reduces the topological symmetry group from $D_4$ to $\Z_4$, as desired.
\end{proof}

%%%%%%%%%%%%%%%%%%%%%%%%%%%%%%%%%%
\section{The graph $K_{3,3,1}$} 
\label{S:K331}

%%%%%%%%%%%%%%%%%%%%%%%%%%%%%%%%%%%%%
 The Petersen graph $K_{3,3,1}$ is the complete tripartite graph with two sets of three vertices labeled $\{1,2,3\}$, $\{4,5,6\}$ and one singleton vertex labeled $\{7\}$, as shown in Figure \ref{F:K331}.    
 
 \begin{figure} [h]
$$\scalebox{.7}{\includegraphics{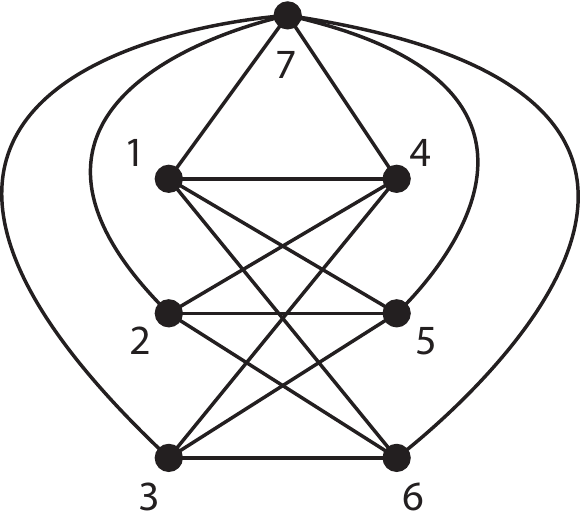}}$$
\caption{The graph $K_{3,3,1}$.}
\label{F:K331}
\end{figure}
 
 %%%%%%%%%%%%%%%%%%%%%%%%%%%%%%%%

In order to determine which groups are realizable for the graph $K_{3,3,1}$, we first observe that $K_{3,3,1}$ contains the graph $K_{3,3}$ as a subgraph (induced by the set of vertices $\{1,2,3,4,5,6\}$) and every automorphism of $K_{3,3,1}$ fixes the subgraph $K_{3,3}$ setwise.  Further, since the remaining vertex 7 is fixed by every automorphism of $K_{3,3,1}$, we conclude that $\Aut(K_{3,3,1}) = \Aut(K_{3,3})$. 
 
If $\sigma$ is a realizable automorphism of $K_{3,3,1}$ then there exists a homeomorphism $h: S^3\to S^3$ taking some embedding $\Gamma$ of $K_{3,3,1}$ to itself which realizes $\sigma$.  Let $\Gamma'$ denote the embedding of $K_{3,3}$ induced by $\Gamma$.  Then $h$ also takes $\Gamma'$ to itself and realizes the automorphism $\sigma \vert_{K_{3,3}}$.  Since Nikkuni and Taniyama's list \cite{nt} gives all realizable automorphisms of $K_{3,3}$ (see Section \ref{SS:K33subgps}), every realizable automorphism of $K_{3,3,1}$ is an extension of one of these.  However, not all automorphisms which are realizable for $K_{3,3}$ are realizable for $K_{3,3,1}$.

%%%%%%%%%%%%%%%%%%%%%%%%%%%%%%%%%%%%%

\begin{lemma}
\label{no3cycle}
No automorphism conjugate to $(123)$ is realizable for any embedding of $K_{3,3,1}$. 
\end{lemma}

    \begin{proof}
First, fix some embedding $\Gamma'$ of $K_{3,3,1}$.  Proceeding by contradiction, assume that $(123)$ is realizable for $\Gamma'$. Then $(123)$ is induced by a homeomorphism $h'$ of $(S^3,\Gamma')$.  Since $h'$ also induces the automorphism $(123)$ on the subgraph $K_{3,3}$, it must be orientation-preserving \cite{nt}. Furthermore, vertices $4$,$5$,$6$, and $7$ are fixed, and so are the edges between them. 

By the Finite Order Theorem, there exists a re-embedding of $\Gamma'$, say $\Gamma$, such that $(123)$ is induced by a finite order homeomorphism $h$ of $(S^3, \Gamma)$. Moreover, since $h$ is a finite order homeomorphism and is orientation preserving, by Smith Theory, $\fix(h)=\emptyset$ or is homeomorphic to $S^1$. However, the subgraph composed of $\{4,5,6,7\}$ and edges $\{4,7\}$,$\{5,7\}$, and $\{6,7\}$, which has a $Y$ shape, is contained in $\fix(h)$, but does not embed in a circle.  This contradicts Smith Theory, therefore $(123)$ is not positively realizable. This implies that no automorphism conjugate to $(123)$ is realizable either.
    \end{proof}
    
%%%%%%%%%%%%%%%%%%%%%%%%%%%%%%%%%%%%%%

\begin{lemma}
\label{no6cycle}
    No order 6 automorphism of $K_{3,3,1}$ is realizable.
\end{lemma}
    \begin{proof}
        Assume there exists an automorphism in $\Aut(K_{3,3,1})$ of order $6$ which is induced by a homeomorphism $h':S^3\rightarrow S^3$ for some embedding $\Gamma'$ of $K_{3,3,1}$. Note that up to conjugation the only order $6$ automorphism of $K_{3,3,1}$ is $(142536)$. Since $h$ also realizes the automorphism $(142536)$ on the subgraph $K_{3,3}$, it must be orientation preserving \cite{nt}. Furthermore, by the Finite Order Theorem, there exists another embedding $\Gamma$ such that $(142536)$ is induced by a finite order orientation preserving homeomorphism $h:S^3\rightarrow S^3$.  Observe, Smith Theory implies that $\fix{(h)}=\emptyset$ or is homeomorphic to $S^1$. But since vertex $7 \in \fix{(h)}$, it follows that $\fix{(h)}$ is $S^1$. 
        
        Next, consider $h^3$, which induces the automorphism $(15)(43)(26)$. Notice, this automorphism fixes the midpoints of the edges $\{1,5\}$, $\{4,3\}$ and $\{2,6\}$. Since $\fix(h) \subseteq \fix{(h^3)} \subseteq S^1 = \fix(h)$, it follows that $\fix(h^3) = \fix(h) = S^1$.  Since $h^3$ fixes the midpoints of the edges $\{1,5\}$, $\{4,3\}$ and $\{2,6\}$, so must $h$. Let us consider then the orbit of the edge $\{1,5\}$ under $h$: $\{\{3,4\},\{2,6\},\{1,5\}\}$. Notice, $h$ must fix the midpoint of edge $\{1,5\}$ under the action of the rotation inducing $(142536)$. But $h$ takes edge $\{1,5\}$ to edge $\{3,4\}$, fixing the midpoint of both, which implies these edges intersect in the embedding $\Gamma$.  This is a contradiction, so no order 6 automorphism is realizable for any embedding of $K_{3,3,1}$. 
    \end{proof}

%%%%%%%%%%%%%%%%%%%%%%%%%%%%%%%%%%%%%

From the Subgraph Lemma (using the subgraph $K_{3,3}$) and Theorem \ref{T:K33TSG+} the possible topological symmetry groups for embeddings of $K_{3,3,1}$ are the following: $D_6$, $D_4$, $D_3$, $D_2$, $\mathbb{Z}_6$, $\Z_4$, $\mathbb{Z}_3$,  $\mathbb{Z}_2$, $D_3\times D_3$, $D_3 \times \mathbb{Z}_3$, $\mathbb{Z}_3 \times \mathbb{Z}_3$, $(D_3 \times D_3) \semi \Z_2$, $(\Z_3 \times \Z_3) \semi \Z_4$, $(\mathbb{Z}_3 \times \mathbb{Z}_3) \rtimes \mathbb{Z}_2$. 

 \begin{lemma}
 \label{Z3Z3}
   The group $\mathbb{Z}_3\times \mathbb{Z}_3$ is not positively realizable for any embedding of $K_{3,3,1}$.
\end{lemma}
    \begin{proof}
        The group $\mathbb{Z}_3 \times \mathbb{Z}_3$ must be generated from an ordered pair of order $3$ automorphisms $\alpha$ and $\beta$ so that $\langle \alpha,\beta \rangle = \mathbb{Z}_3 \times \mathbb{Z}_3$. Since automorphisms conjugate to $(123)$ are not realizable by Lemma \ref{no3cycle}, the only possibilities for $\alpha$ and $\beta$ are $(123)(456)$ and its conjugates $(123)(465)$, $(132)(456)$, and $(132)(465)$. Each pair of these automorphisms either are inverses and together generate $\mathbb{Z}_3$, or the product of $\alpha$ and $\beta$ gives an automorphism conjugate to $(123)$ which is not realizable by Lemma \ref{no3cycle}. Thus $\mathbb{Z}_3 \times \mathbb{Z}_3$ is not realizable for any embedding of $K_{3,3,1}$. 
    \end{proof}
    
In fact, this argument allows us to conclude that any group that contains $\Z_3 \times \Z_3$ as a subgroup is also not realizable. Hence $D_3\times D_3$, $D_3 \times \mathbb{Z}_3$, $(D_3 \times D_3) \semi \Z_2$, $(\Z_3 \times \Z_3) \semi \Z_4$, and $(\mathbb{Z}_3 \times \mathbb{Z}_3) \rtimes \mathbb{Z}_2$ are also not positively realizable for any embedding of $K_{3,3,1}$. 

From Lemma \ref{no6cycle}, a topological symmetry group for $K_{3,3,1}$ cannot contain a 6-cycle, so $D_6$ and $\Z_6$ are also not realizable for $K_{3,3,1}$.  This reduces our list of possibilities for realizable topological symmetry groups of $K_{3,3,1}$ to $D_4$, $D_3$, $D_2$, $\Z_4$, $\mathbb{Z}_3$, and $\mathbb{Z}_2$. 
   
    \begin{theorem}
\label{K331tsg}
            A nontrivial group is positively realizable for  $K_{3,3,1}$ if and only if it is one of: $D_3$, $D_2$, $\mathbb{Z}_3$, and $\mathbb{Z}_2$.
\end{theorem}

\begin{proof}  
Since $D_4$ and $\Z_4$ are not positively realizable for $K_{3,3}$ (by Theorem \ref{T:K33TSG+}), they are not positively realizable for $K_{3,3,1}$ by the Subgraph Lemma.  It only remains to show that $D_3$, $D_2$, $\mathbb{Z}_3$, and $\mathbb{Z}_2$ are positively realizable.  First consider the embedding $\Gamma_1$ in Figure \ref{F:K331D3}.  Let $h$ denote a $\frac{2\pi}{3}$ rotation about the axis perpendicular to the plane and passing through vertex $7$.  Then $h$ induces the automorphism $(123)(456)$.  Let $g$ denote the rotation of $180^{\circ}$ around the axis bisecting edges $\{1,4\}$ and $\{3,5\}$ and passing through vertex $7$. Then $g$ induces the automorphism $(14)(26)(35)$.  It follows that $\TSG_+(\Gamma_1)$ contains the subgroup $\langle h, g \  | \ hg=gh^{-1} \rangle \cong D_3$. But the only remaining possible topological symmetry group that contains $D_3$ is $D_3$ itself.  Hence $\TSG_+(\Gamma_1) \cong D_3$.  Since the edge $\{1,6\}$ is not fixed by any element of $\TSG_+(\Gamma_1)$, by the Subgroup Theorem there is another embedding $\Gamma_1'$ for which $\TSG_+(\Gamma_1') \cong \Z_3$.

Now, consider the embedding $\Gamma_2$ of $K_{3,3,1}$ given in Figure \ref{F:K331D2}. The arrows pointing outward from the hexagon denote the edges connected to vertex $7$, which is fixed to the point at infinity. Now, we induce the automorphism $(14)(25)(36)$ by rotating the hexagon by $\pi$, and we induce the automorphism $(12)(45)$ by rotating the hexagon about the axis containing the edge $\{3,6\}$. These two automorphisms generate a group isomorphic to $D_2$, so $D_2 \leq \TSG_+(\Gamma_2)$. But the only remaining possible topological symmetry group that contains $D_2$ is $D_2$ itself. Therefore, $\TSG_+(\Gamma_2) \cong D_2$. Since the edge $\{3,6\}$ is, once again, not fixed by any element of the topological symmetry group, the Subgroup Theorem tells us there is another embedding which positively realizes $\Z_2$.
\end{proof}

    \begin{figure} [htbp]
$$\scalebox{.8}{\includegraphics{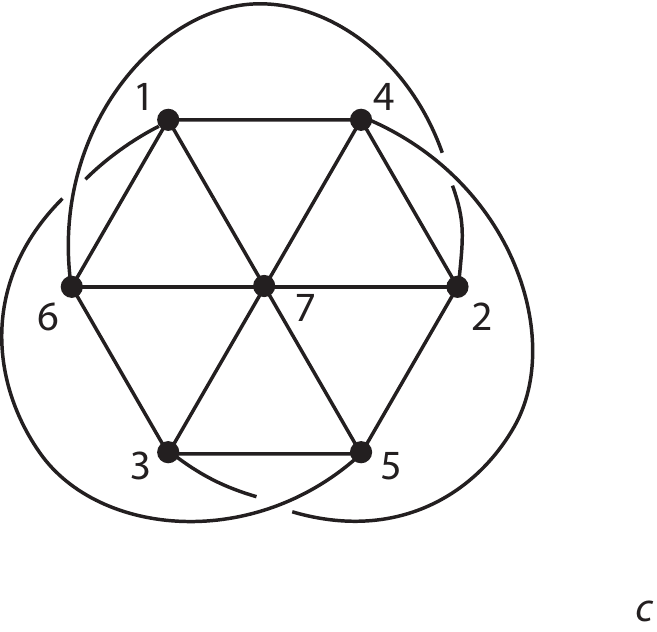}}$$
\caption{An embedding $\Gamma_1$ of the graph $K_{3,3,1}$ with $ \TSG_+(\Gamma_1) = D_3$.}
\label{F:K331D3}
\end{figure}

    \begin{figure} [htbp]
$$\scalebox{.45}{\includegraphics{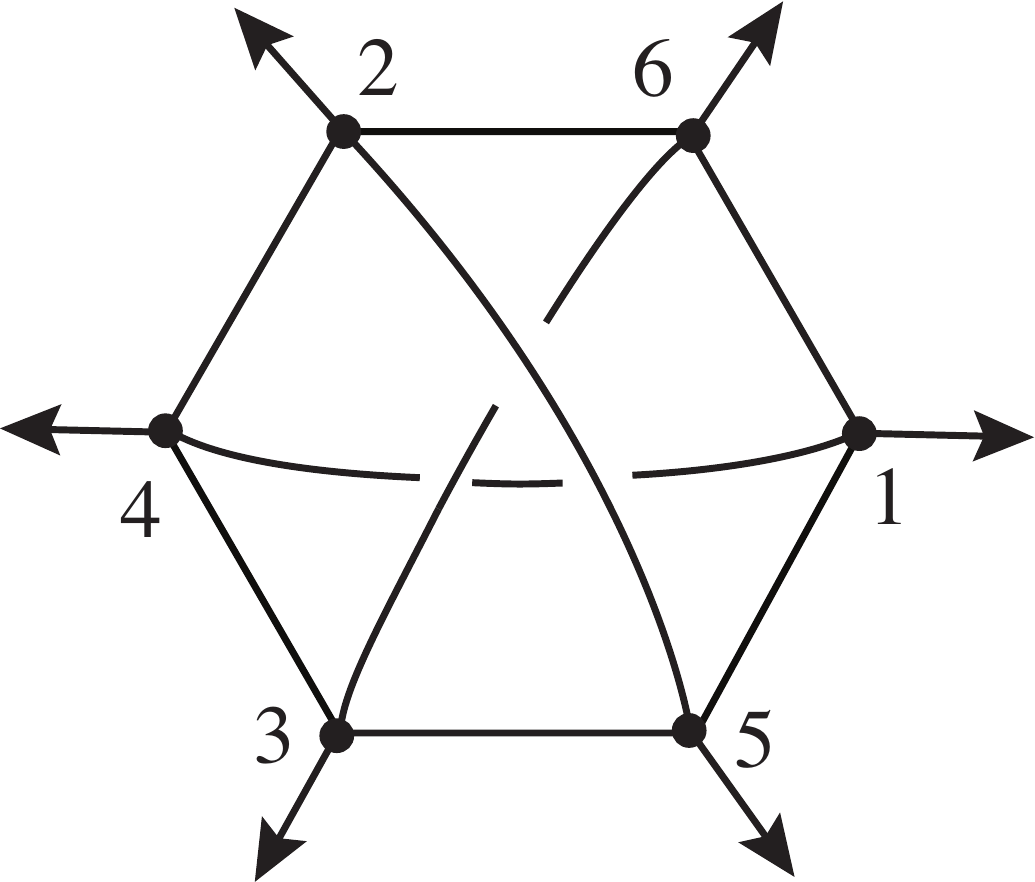}}$$
\caption{An embedding $\Gamma_2$ of the graph $K_{3,3,1}$ with $ \TSG_+(\Gamma_2) = D_2$. Vertex 7 is at $\infty$.}
\label{F:K331D2}
\end{figure}

\begin{theorem} \label{T:K331TSG}
The groups which are realizable but {\bf not} positively realizable for $K_{3,3,1}$ are $D_4$ and $\Z_4$.
\end{theorem}

\begin{proof}
We have already dealt with all other subgroups of $\Aut(K_{3,3,1})$, so it only remains to show that $D_4$ and $\Z_4$ are realizable.  Let $\Gamma_3$ of $K_{3,3,1}$, be the embedding in Figure \ref{F:K331D4}.  Here we are viewing the embedding in $\mathbb{R}^3$, together with the point at infinity (to give $S^3$). The vertices $3$ and $6$ are on a vertical straight line; the vertex $7$ is also on this line, at the point at infinity, denoted $\infty$.  The vertices $\{1, 3, 2\}$ are arranged along a straight line perpendicular to the line through vertices $\{3,6,\infty\}$ at vertex $3$, while the vertices $\{4, 6, 5\}$ are arranged along a straight line perpendicular to the line through vertices $\{3,6,\infty\}$ at vertex $6$, which is at right angles to the line containing $\{1, 3, 2\}$.  Let $h$ denote the homeomorphism obtained by rotating the axis containing $\{3,6,\infty\}$ by $\frac{\pi}{2}$ followed by a reflection about the plane perpendicular to this axis and bisecting edge $\{3,6\}$. Then $h$ induces the order $4$ automorphism $(1425)(36)$.  Let $g$ be the reflection of $\Gamma_3$ through the plane containing vertices $\{3,4,5,6, \infty\}$.  Then $g$ induces the automorphism $(12)$.  Therefore $\langle h,g \ | \ hg=gh^{-1} \rangle \cong D_4 \leq TSG(\Gamma_2)$.  However, the only proper subgroup of $\Aut(K_{3,3,1})$ that contains $D_4$ is $D_4$ itself, so this embedding has a topological symmetry group isomorphic to $D_4$.

To obtain an embedding $\Gamma_3$ with the symmetry group $\Z_4$, we add  the knot $8_{17}$ to edges $\{1,4\}$, $\{1,6\}$, $\{2,4\}$, and $\{2,6\}$.  The permutation $(12)$ must be realized by an orientation reversing homeomorphism taking $8_{17}$ to its mirror image which is impossible.  Therefore this embedding does not allow for $(12)$.  However it does allow for the permutation $(1425)(36)=(12)(14)(25)(36)$, realized by the homeomorphism $h$ above. This homeomorphism takes $8_{17}$ to the mirror image of its inverse; since $8_{17}$ is a negative amphicheiral knot, this is equivalent to the original knot.   Thus, $TSG(\Gamma_3)$ is reduced in this embedding to $\Z_4$.  
\end{proof}

    \begin{figure} [htbp]
$$\scalebox{.7}{\includegraphics{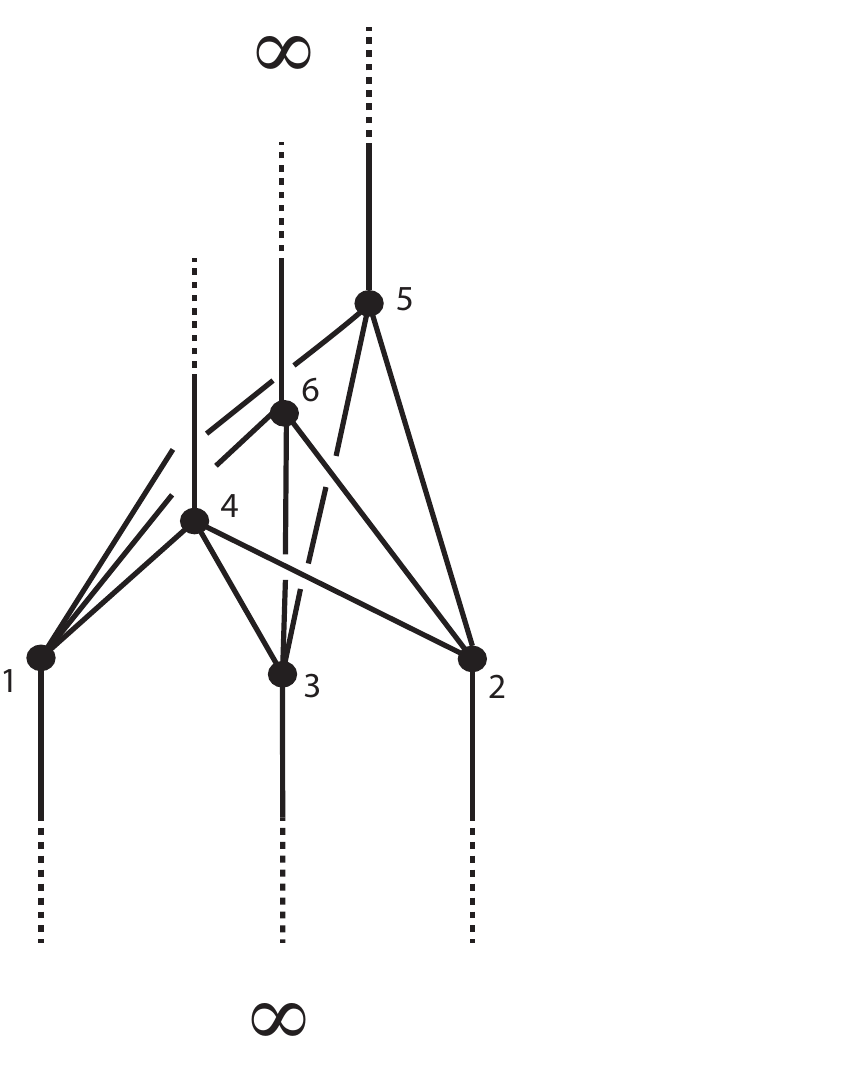}}$$
\caption{An embedding $\Gamma_3$ of the graph $K_{3,3,1}$ with $ \TSG_+(\Gamma_3) = D_4$. Vertex 7 is at $\infty$.}

\label{F:K331D4}
\end{figure}

%%%%%%%%%%%%%%%%%%%%%%%%%%%%%%%%%%

\section{The graph $K_{4,4}^-$} 
\label{S:K44}

%%%%%%%%%%%%%%%%%%%%%%%%%%%%%%%%%%%%%
 The Petersen graph $K_{4,4}^-$ is the complete bipartite graph $K_{4,4}$ with an edge removed, with vertex sets labeled $\{1,2,3\}$, $\{4,5,6\}$, and $\{v,w\}$, as shown in Figure \ref{F:K44}.  
 
 \begin{figure} [h]
$$\scalebox{.8}{\includegraphics{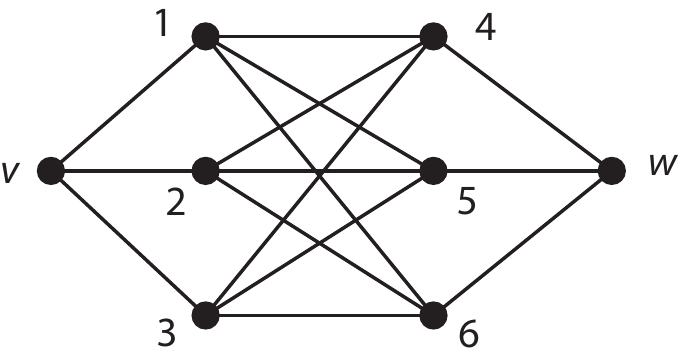}}$$
\caption{The graph $K_{4,4}^-$.}
\label{F:K44}
\end{figure}
 
 %%%%%%%%%%%%%%%%%%%%%%%%%%%%%%%%

Like $K_{3,3,1}$, $K_{4,4}^-$ contains a $K_{3,3}$ subgraph (induced by the vertices $\{1, 2, 3, 4, 5, 6\}$) which is fixed by every automorphism of the graph.  Moreover, an automorphism of $K_{4,4}^-$ interchanges $v$ and $w$ if it interchanges the sets $\{1, 2, 3\}$ and $\{4, 5, 6\}$, and fixes $v$ and $w$ otherwise.  So $\Aut(K_{4,4}^-) \cong \Aut(K_{3,3})$. Hence, as with $K_{3,3,1}$, every realizable automorphism of $K_{4,4}^-$ must also be realizable for $K_{3,3}$.  The converse, however, is not true. As with $K_{3,3,1}$, automorphisms conjugate to $(123)$ are not realizable.  The proof is almost identical to Lemma \ref{no3cycle}, and the details are left to the reader.
 
\begin{lemma}
\label{no3cycle2}
No automorphism conjugate to $(123)$ is realizable for any embedding of $K_{4,4}^-$. 
\end{lemma}

As for $K_{3,3,1}$, this means that no group that contains $\Z_3 \times \Z_3$ is realizable for $K_{4,4}^-$; the proof is the same as Lemma \ref{Z3Z3}.  We can now reduce our list of possible topological symmetry groups for embeddings of $K_{4,4}^-$ to the following: $D_6$, $D_4$, $D_3$, $D_2$, $\mathbb{Z}_6$, $\Z_4$, $\mathbb{Z}_3$,  $\mathbb{Z}_2$. 

\begin{theorem}\label{K44tsg}
A nontrivial group is positively realizable for $K_{4,4}^-$ if and only if it is one of:  $D_6$, $D_3$, $D_2$, $\mathbb{Z}_6$, $\Z_3$ and $\mathbb{Z}_2$.  \end{theorem}
\begin{proof}
Since $D_4$ and $\Z_4$ are not positively realizable for $K_{3,3}$, they are not positively realizable for $K_{4,4}^-$ by the Subgraph Lemma. It remains to show that the other groups are positively realizable. Let $\Gamma_1$ be the embedding of $K_{4,4}^-$ in Figure \ref{F:K44edgeD6} where vertices $v$ and $w$ are placed at antipodal points of a geodesic circle and equidistant from the plane through the midpoints of edges $\{1,5\}$, $\{3,4\}$ and $\{2,6\}$.  This embedding has a glide rotation $h$ obtained by rotating the picture by  $\frac{2\pi}{3}$ around the axis going through vertices $v$ and $w$ while rotating by $\pi$ around the circular waist of the picture.  Then $h$ induces the order $6$ automorphism $(142536)(vw)$.  Consider the line perpendicular to the axis containing $v$ and $w$ and passing through the midpoint of $\{3,4\}$.  A rotation by $\pi$ about this line is a homeomorphism $g$ that induces $(16)(25)(34)(vw)$.  Since $hg=gh^{-1}$ it follows that $D_6 \leq \TSG_+(\Gamma_1)$; but since $D_6$ is the largest remaining possible group, this means $\TSG_+(\Gamma_1) \cong D_6$.  Since the edge $\{1,4\}$ is not fixed by any element of this group, the Subgroup Theorem implies that every subgroup of $D_6$ is also positively realizable for some embedding of $K_{4,4}^-$. 
\end{proof}

    \begin{figure} [htbp]
$$\scalebox{.6}{\includegraphics{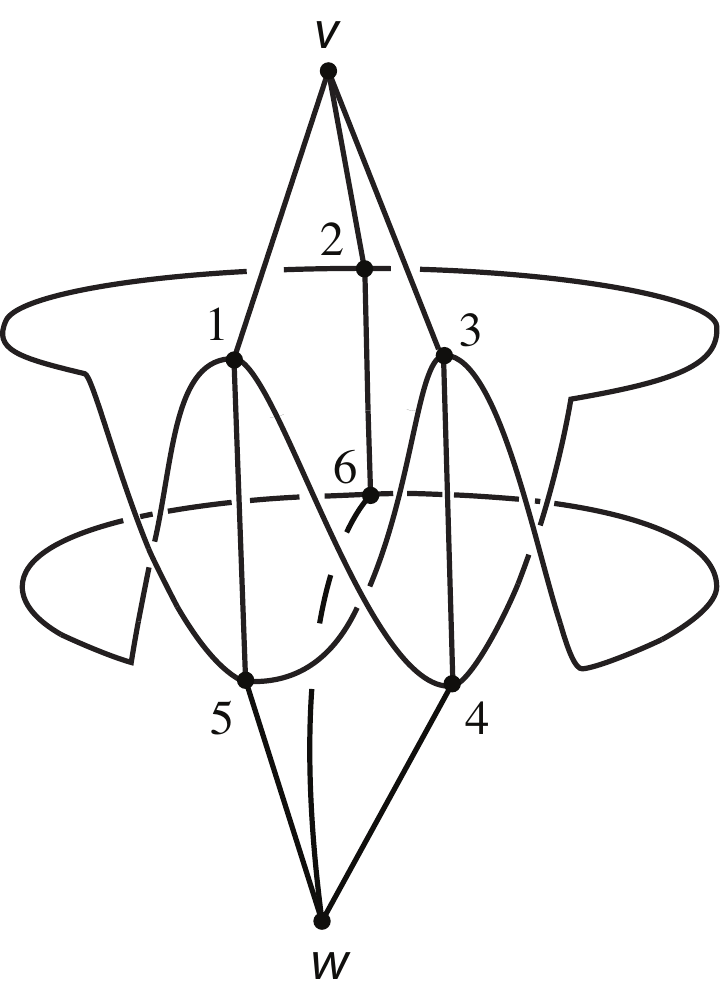}}$$
\caption{An embedding $\Gamma_1$ of the graph $K_{4,4}^-$ with $ \TSG_+(\Gamma_2) = D_6$.}
\label{F:K44edgeD6}
\end{figure}

%%%%%%%%%%%%%%%%%%%%%%%%%%%%%%%%%

\begin{theorem} \label{T:K44edgeTSG2}
The groups which are realizable but {\bf not} positively realizable for $K_{4,4}^-$ are $D_4$ and $\Z_4$.
\end{theorem}

\begin{proof}
It only remains to show that $D_4$ and $\Z_4$ are realizable.  Let $\Gamma_2$ be the embedding of $K_{4,4}^-$ in Figure \ref{F:K44edgeD4}.  The vertices $\{1, 2, 3, 4, 5, 6\}$ are arranged as in Figure \ref{F:K331D4}; the vertices $v$ and $w$ are placed symmetrically on the axis containing  the edge $\{3,6\}$.  Let $h$ denote the homeomorphism obtained by rotating the axis containing $\{3,6,v,w\}$ by $\frac{\pi}{2}$ followed by a reflection about the plane perpendicular to this axis and bisecting edge $\{3,6\}$. Then $h$ induces the order $4$ automorphism $(1425)(36)(vw)$.  Let $g$ be the reflection of $\Gamma_2$ through the plane containing vertices $\{3,4,5,6,v,w\}$.  Then $g$ induces the automorphism $(12)$.  Therefore $\langle h,g \ | \ hg=gh^{-1} \rangle \cong D_4 \leq TSG(\Gamma_2)$.  However, the only proper subgroup of $\Aut(K_{3,3,1})$ that contains $D_4$ is $D_4$ itself, so this embedding has a topological symmetry group isomorphic to $D_4$.

To obtain an embedding $\Gamma_3$ with the symmetry group $\Z_4$, as above, we add knot $8_{17}$ to edges $\{1,4\}$, $\{1,6\}$, $\{2,4\}$, and $\{2,6\}$.  The permutation $(12)$ must be realized by an orientation reversing homeomorphisms which takes $8_{17}$ to its mirror image which is impossible.  Therefore this embedding does not allow for $(12)$.  However it does allow for the permutation $(1425)(36)(vw)=(12)(14)(25)(36)(vw)$, generated by $h$. This takes $8_{17}$ to itself in each edge (since $8_{17}$ is negative amphicheiral). Thus, $\TSG(\Gamma_3)$ is reduced to $\Z_4$.  
\end{proof}

    \begin{figure} [htbp]
$$\scalebox{.7}{\includegraphics{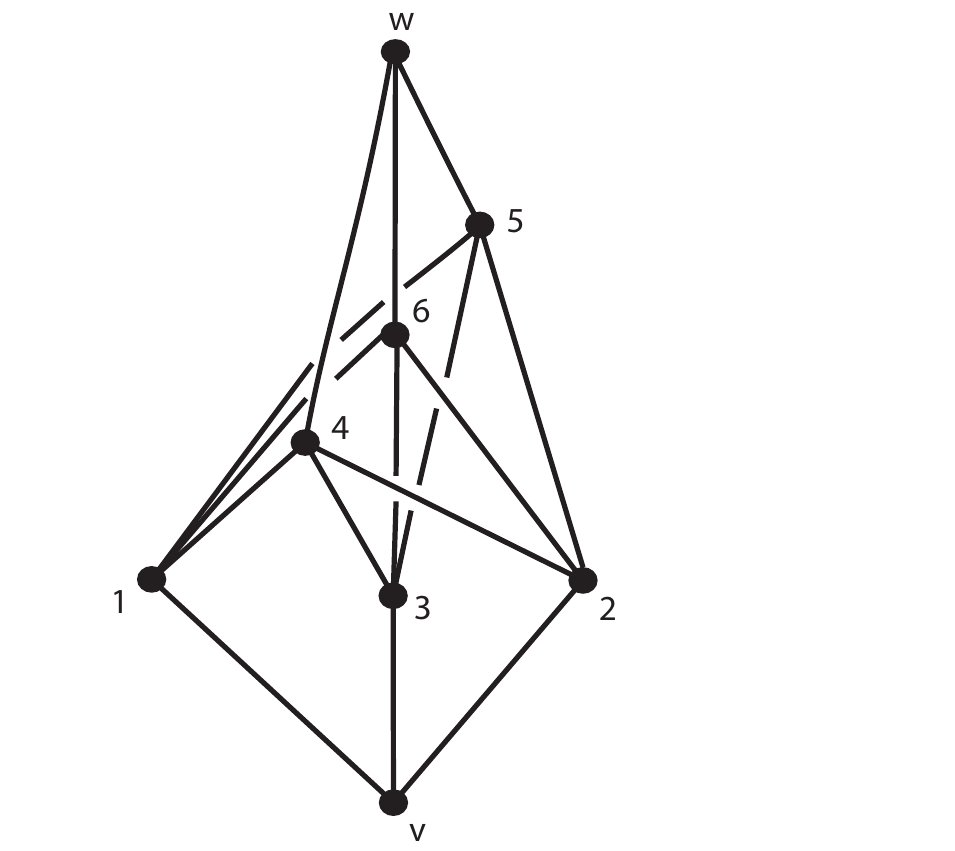}}$$
\caption{An embedding $\Gamma_2$ of the graph $K_{4,4}^-$ with $ \TSG_+(\Gamma_2) = D_4$.}
\label{F:K44edgeD4}
\end{figure}

%%%%%%%%%%%%%%%%%%%%%%%%%%%%%%%%%%
\section{The graph $P_7$}\label{S:P7}

Consider the graph $P_7$ with the vertices labeled as in Figure \ref{F:P7}. We begin by determining the automorphism group of the abstract graph $P_7$.

\begin{figure} [htbp]
$$\scalebox{.8}{\includegraphics{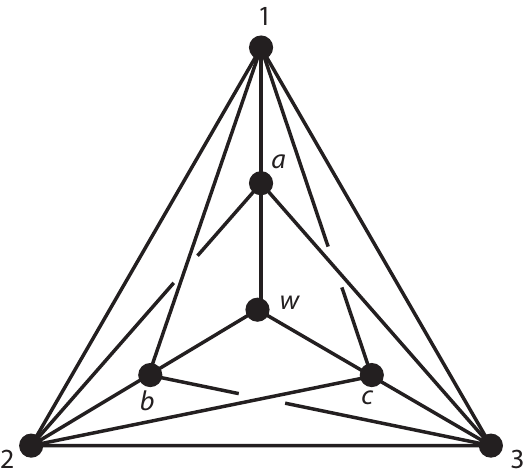}}$$
\caption{An embedding $\Gamma$ of the graph $P_7$ with $\TSG_+(\Gamma) = D_3$.}
\label{F:P7}
\end{figure}

\begin{theorem}\label{T:P7aut}
The automorphism group $\Aut(P_7)$ is isomorphic to $D_3 \times D_3$.
\end{theorem}
\begin{proof}
As we can see from Figure \ref{F:P7}, $P_7$ has three vertices (labeled 1, 2, 3) of degree 5, three vertices (labeled $a, b, c$) of degree 4, and one vertex (labeled $w$) of degree 3. So every automorphism must fix each of these three sets (setwise). Moreover, the vertices in each set have the same neighbors (and within each set, the vertices are either all adjacent, or all independent), so the vertices in each set can be permuted independently.  Hence $\Aut(P_7) = S_3 \times S_3 \times S_1 \cong D_3 \times D_3$.
\end{proof}

The subgroups of $D_3 \times D_3$ are $D_3 \times D_3$, $D_3 \times \Z_3$, $D_3 \times \Z_2$, $(\Z_3 \times \Z_3) \rtimes \Z_2$, $\Z_3 \times \Z_3$, $\Z_3 \times \Z_2$, $\Z_2 \times \Z_2$, $D_3$, $\Z_3$, $\Z_2$ and the trivial group (this can easily be checked with a program such as GAP). We make two important observations: \begin{enumerate}
	\item An automorphism of $P_7$ that fixes the vertices 1, 2 and 3 will fix the 3-cycle through these vertices {\em and} the vertex $w$.  Hence, by Smith Theory, it cannot be realized by a non-trivial orientation-preserving homeomorphism.
	\item An automorphism of $P_7$ that fixes the vertices $a$, $b$ and $c$ will fix the subgraph induced by $\{a, b, c, w\}$, which does not embed in a circle.  So, again by Smith Theory, such an automorphism cannot be realized by a non-trivial orientation-preserving homeomorphism.
\end{enumerate}

\begin{theorem}\label{T:P7op}
A nontrivial group is positively realizable for $P_7$ if and only if it is one of: $D_3$, $\Z_3$ or $\Z_2$.
\end{theorem}
\begin{proof}
Suppose $\Gamma$ is an embedding of $P_7$, and $H = \TSG_+(\Gamma)$.  Then $H$ is isomorphic to a subgroup of $D_3 \times D_3$, where the first factor corresponds to permutations of $\{1, 2, 3\}$ and the second factor corresponds to permutations of $\{a, b, c\}$ (the vertex $w$ is fixed by every element of $H$). Suppose $(\sigma, \rho) \in H$. If there is a second permutation $\rho'$ such that $(\sigma, \rho') \in H$, then $(\sigma, \rho)(\sigma^{-1}, (\rho')^{-1}) = (id, \rho(\rho')^{-1}) \in H$. But by our observations above, this means $\rho(\rho')^{-1} = id$ as well, so $\rho' = \rho$. So for each permutation $\sigma$, there is at most one permutation $\rho$ such that $(\sigma, \rho) \in H$. So the projection to the first component gives an isomorphism from $H$ to a subgroup of $D_3$. So every positively realizable group is isomorphic to a subgroup of $D_3$.

Conversely, consider the embedding $\Gamma$ of $P_7$ in Figure \ref{F:P7}. $\TSG_+(\Gamma) \cong D_3$, generated by the rotation of order 3 around the axis through vertex $w$ (perpendicular to the page), and the rotation of order 2 around the axis through vertices 1, $a$ and $w$.  Since the edge $\{1,2\}$ is not fixed by any element of this group, the Subgroup Theorem there are embeddings of $P_7$ which positively realize every subgroup of $D_3$.
\end{proof}

Since the subgroups of $D_3$ are positively realizable, they are also realizable. It remains to ask whether any subgroups of $D_3 \times D_3$ are realizable, but not positively realizable.

\begin{figure} [htbp]
$$\scalebox{1}{\includegraphics{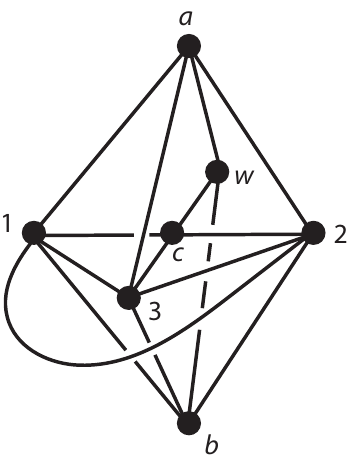}}$$
\caption{An embedding $\Gamma$ of the graph $P_7$ with $\TSG(\Gamma) = \Z_2 \times \Z_2$.}
\label{F:P7Z2xZ2}
\end{figure}

\begin{theorem}\label{T:P7or}
The only group which is realizable, but not positively realizable, for $P_7$ is $\Z_2 \times \Z_2$.
\end{theorem}
\begin{proof}
Suppose $\Gamma$ is an embedding of $P_7$ with $H = \TSG(\Gamma) \neq \TSG_+(\Gamma)$.  Then $\TSG_+(\Gamma)$ is a normal subgroup of $H$ of index 2.  Since $\TSG_+(\Gamma)$ is a subgroup of $D_3$, $\vert \TSG_+(\Gamma) \vert \leq \vert D_3 \vert = 6$, so $\vert H\vert \leq 12$.  Hence $H \neq D_3 \times D_3, D_3 \times \Z_3,$ or $(\Z_3 \times \Z_3) \rtimes \Z_2$. Also, $\vert H\vert$ must be even, so $H \neq \Z_3 \times \Z_3$. It remains to consider $D_3 \times \Z_2$, $\Z_3 \times \Z_2$ and $\Z_2 \times \Z_2$.

Suppose $H$ has an element $h$ of order 6. Since $h$ is not in a subgroup of $D_3$, it must be orientation-reversing. Since $H \leq D_3 \times D_3$, $h = (\sigma, \rho)$ for some $\sigma, \rho \in D_3$.  Without loss of generality, suppose $\sigma$ is a 3-cycle and $\rho$ is a 2-cycle; we can assume $h$ induces the permuation $(123)(ab)(c)(w)$ on the vertices of $P_7$.  Then $h^2$ induces the permutation $(132)(a)(b)(c)(w)$.  But this means $h^2$ is an non-trivial orientation-preserving homeomorphism that fixes the subgraph induced by $\{a, b, c, w\}$, which violates Smith Theory. So $H$ cannot have an element of order 6, which means $H \neq D_3 \times \Z_2$ or $\Z_3 \times \Z_2$. The only possibility that remains is $\Z_2 \times \Z_2$.

Consider the embedding of $P_7$ shown in Figure \ref{F:P7Z2xZ2}. Here the vertices $1, 2, 3, a, b, w$ are placed at the vertices of a regular octahedron, and $c$ is placed at the center of the octahedron. The edge $\{1,2\}$ is in the horizontal plane containing vertices $1, 2, 3, c, w$.  Let $h$ be the reflection in the vertical plane containing the vertices $3, a, b, c, w$, interchanging vertices 1 and 2, and let $g$ be the reflection in the horizontal plane, interchanging vertices $a$ and $b$. Then $hg$ is the rotation of order 2 about the axis through vertices $3, c, w$. So $\TSG(\Gamma) \cong \Z_2 \times \Z_2$.  Hence $\Z_2 \times \Z_2$ is the only group which is realizable, but not positively realizable for $P_7$.
\end{proof}

%%%%%%%%%%%%%%%%%%%%%%%%%%%%%%%%%%
\section{The graph $P_8$}\label{S:P8}

Consider the graph $P_8$ with the vertices labeled as in Figure \ref{F:P8}. We begin by determining the automorphism group of the abstract graph $P_8$.

\begin{figure} [htbp]
$$\scalebox{.7}{\includegraphics{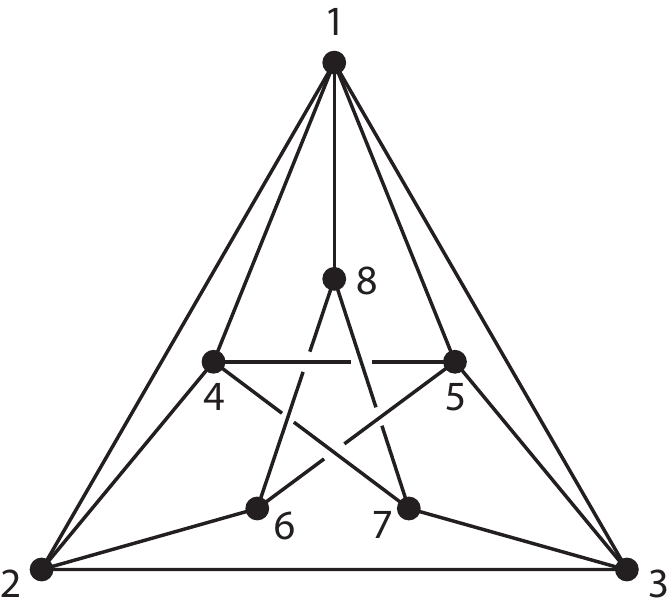}}$$
\caption{An embedding $\Gamma$ of the graph $P_8$ with $\TSG(\Gamma) = \Z_2$.}
\label{F:P8}
\end{figure}

\begin{theorem} \label{T:P8aut}
The automorphism group $\Aut(P_8)$ is isomorphic to $D_4$.
\end{theorem}
\begin{proof}
We first observe that vertex 1 is the only vertex of degree five, and hence is fixed by every automorphism of $P_8$.  Since vertex 8 is the only vertex of degree three which is adjacent to vertex 1, it is also fixed by every automorphism.

The vertices 2, 3, 4, 5 of degree four are fixed as a set by every automorphism of $P_8$, and form a 4-cycle in the graph. Hence the group of automorphisms of this subgraph is isomorphic to $D_4$.  Since vertices 6 and 7 are each adjacent to a different pair of non-adjacent vertices in the 4-cycle, every automorphism of the 4-cycle induces a unique permutation of the set $\{6, 7\}$ (and of course fixes vertices 1 and 8). So every automorphism of the 4-cycle induces a unique automorphism of $P_8$.  Hence $\Aut(P_8) \cong D_4$.
\end{proof}

The subgroups of $D_4$ are $D_4$, $\Z_4$, $D_2$, $\Z_2$ and the trivial group. We will see that only $\Z_2$ and the trivial group are either realizable or positively realizable for $P_8$.

\begin{theorem} \label{T:P8}
A nontrivial group is realizable or positively realizable for $P_8$ if and only if it is $\Z_2$.
\end{theorem}
\begin{proof}
Suppose $\Gamma$ is an embedding of $P_8$. We first observe that any non-trivial element $h \in \TSG(\Gamma)$ must either fix or interchange the vertices 6 and 7.  Since $h$ must also fix vertices 1 and 8, this means that if $h$ fixes 6 and 7, then it fixes the subgraph induced by $\{1, 6, 7, 8\}$.  This subgraph does not embed in a circle, so by Smith Theory $h$ cannot be an orientation-preserving homeomorphism of $S^3$. So any orientation-preserving element of $\TSG(\Gamma)$ must interchange vertices 6 and 7.

But for any non-trivial $h \in \TSG(\Gamma)$, $h^2$ is an orientation preserving homeomorphism which fixes vertices 6 and 7 (since $h$ either fixes or interchanges them).  So $h^2$ must be the identity.  Hence every non-trivial element of $\TSG(\Gamma)$ has order 2. This means $\TSG(\Gamma)$ cannot be isomorphic to $D_4$ or $\Z_4$.

If $h$ and $g$ are both non-trivial orientation-preserving elements of $\TSG(\Gamma)$, then they both interchange vertices 6 and 7, so the product $hg$ is an orientation-preserving homeomorphism of $S^3$ that fixes vertices 1, 6, 7, 8.  Hence $hg$ is the identity, so $g= h^{-1} = h$.  Hence $\TSG_+(\Gamma) \leq \Z_2$.

Suppose $\TSG(\Gamma) \cong D_2 \cong \Z_2 \times \Z_2$.  This group has three non-trivial elements; since the product of two orientation-reversing homeomorphisms is orientation-preserving, at least one of the three is orientation-preserving.  And by the previous paragraph, at most one is orientation-preserving. So there are two orientation-reversing homeomorphisms $h$ and $g$, and the product $hg$. Since $hg$ is orientation-preserving, it interchanges vertices 6 and 7; hence one of $h$ or $g$ must interchange these vertices and the other one fixes them. But only two automorphisms of $P_8$ interchange vertices 6 and 7 and have order two: the permutations $(1)(8)(67)(23)(45)$ and $(1)(8)(67)(24)(35)$.  So two of the elements of $\TSG(\Gamma)$ induce these permutations on the vertices, and the third induces the product $(1)(8)(6)(7)(25)(34)$. In particular, the homeomorphism $h$ that induces $(1)(8)(6)(7)(25)(34)$ is orientation-reversing, and by Smith Theory has a fixed point set homeomorphic to $S^2$.  So vertices 2 and 5 are on opposite sides of the sphere, as are vertices 3 and 4. Without loss of generality, suppose vertices 2 and 3 are on one side of the sphere, and 4 and 5 are on the other side.  Then the edges $\{2,4\}$ and $\{3,5\}$ must pass through the sphere. So $h$ fixes a point on each of these edges, and hence must fix these edges setwise. But this contradicts that $h(\{2,4\}) = \{3,5\}$, since edges $\{2,4\}$ and $\{3,5\}$ are disjoint. Hence, it is impossible for $\TSG(\Gamma) \cong D_2$.

We conclude that $\TSG(\Gamma) \leq \Z_2$ as well.  It remains to show that we can realize $\Z_2$.  Consider the embedding $\Gamma$ of $P_8$ shown in Figure \ref{F:P8}, and the vertical axis through the vertices 1 and 8.  A half turn around this axis is an orientation-preserving symmetry of $\Gamma$ (realizing the permutation $(1)(8)(67)(23)(45)$). This is orientation-preserving, so $\TSG(\Gamma) = \TSG_+(\Gamma) \cong \Z_2$.
\end{proof}

%%%%%%%%%%%%%%%%%%%%%%%%%%%%%%%%%%
\section{The graph $P_9$}\label{S:P9}

We consider the graph $P_9$ with the vertices labeled as in Figure \ref{F:P9}. To begin, we determine the automorphism group of the abstract graph $P_9$.

\begin{figure} [h]
$$\scalebox{.8}{\includegraphics{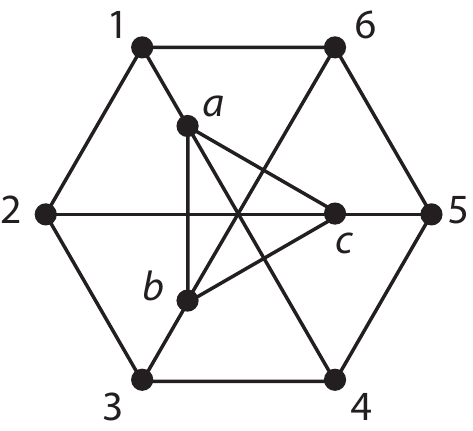}}$$
\caption{The graph $P_9$.}
\label{F:P9}
\end{figure}

\begin{theorem} \label{T:P9aut}
The automorphism group of $P_9$ is isomorphic to $D_6$.
\end{theorem}
\begin{proof}
We first observe from Figure \ref{F:P9} that there are six vertices of degree three, labeled 1 through 6, and three vertices of degree four, labeled $a, b, c$. Every automorphism of $P_9$ must fix these two sets of vertices (setwise). The subgraph $H$ induced by vertices 1 through 6 is a hexagon, so $\Aut(H) = D_6$.  Every automorphism of $H$ permutes the pairs of antipodal vertices: $\{1,4\}$, $\{2,5\}$ and $\{3,6\}$.  Since the three vertices $a, b, c$ are each adjacent to the two vertices in one of these pairs (and none of the vertices in the other pairs), the permutation of these pairs determines the permutation of $a, b, c$. In other words, each automorphism of $H$ induces a unique automorphism of $P_9$.  Hence $\Aut(P_9) \cong \Aut(H) = D_6$.
\end{proof}

So every possible topological symmetry group for $P_9$ is isomorphic to a subgroup of $D_6$. We first show there is an embedding of $P_9$ whose orientation-preserving topological symmetry group is $D_6$ itself.

\begin{theorem}\label{T:P9}
A nontrivial group is realizable or positively realizable for $P_9$ if and only if it is one of: $D_6$, $D_3$, $D_2$, $\Z_6$, $\Z_3$, or $\Z_2$.
\end{theorem}
\begin{proof}
We will show that every subgroup of $\Aut(P_9) = D_6$ is positively realizable, and hence also realizable. Consider the embedding $\Gamma$ shown in Figure \ref{F:P9D6}.  Here the cycle $\overline{123456}$ is embedded as a great circle of $S^3$, and the cycle $\overline{abc}$ is embedded as the great circle which is perpendicular to the disk bounded by the first great circle at the center of the disk. The vertices are equally spaced around each of the great circles.

\begin{figure} [h]
$$\scalebox{.7}{\includegraphics{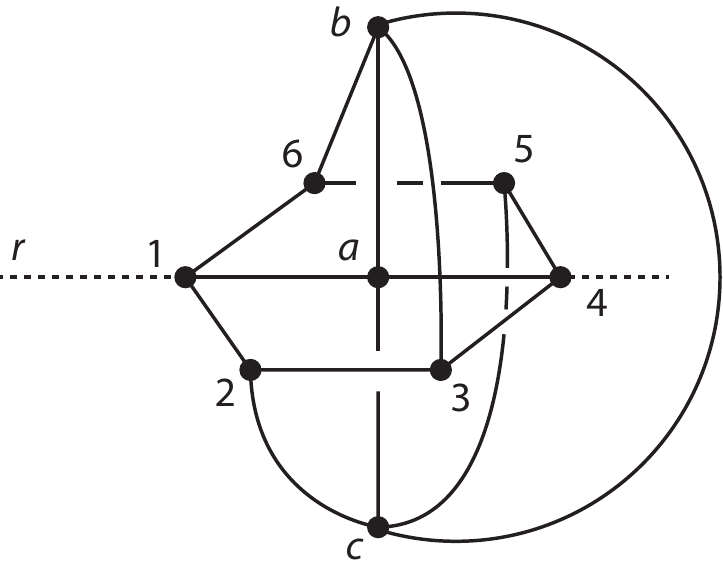}}$$
\caption{An embedding $\Gamma$ of $P_9$ with $\TSG(\Gamma) = D_6$.}
\label{F:P9D6}
\end{figure}

We consider the following two motions: first, the glide rotation $g$ which (1) rotates around the axis $\overline{abc}$ by 1/6 of a full turn, and (2) rotates around the axis $\overline{123456}$ by 1/3 of a full turn; and secondly, the rotation $r$ by a half turn around the axis through vertices $1$, $a$ and $4$ (as shown in Figure \ref{F:P9D6}).  Both $g$ and $r$ are orientation-preserving homeomorphisms. So $g$ induces the permutation $(abc)(123456)$ on the vertices, and $r$ induces the permutation $(a)(bc)(1)(4)(26)(35)$. It is easy to check that $gr = rg^{-1} = (ab)(c)(12)(36)(45)$, so these two motions generate a group isomorphic to $D_6$. Since $\TSG_+(\Gamma)$ must be a subgroup of $\Aut(P_9) = D_6$, this means $\TSG_+(\Gamma) \cong D_6$.

Since the edge $\{1,2\}$ is not pointwise fixed by any non-trivial element of the topological symmetry group, applying the Subgroup Theorem to $\Gamma$ shows that every subgroup of $D_6$ is also positively realizable.
\end{proof}

%%%%%%%%%%%%%%%%%%%%%%%%%%%%%%%%%%%%%%%%%%
\section{Conclusion and Future Work}

Table \ref{Ta:TSG} summarizes our results, listing all groups which can be positively realized or realized for each of the graphs in the Petersen family.

\begin{table}[htbp]
\begin{center}
\begin{tabular}{|c|c|c|}
\hline
{\bf Graph} & {\bf Positively Realizable} & {\bf Realizable ({\em not} positively)} \\ \hline
$K_6$ & \parbox{2.5in}{\rule{0in}{.15in} $D_6, D_5, D_3, D_2, \Z_6, \Z_5, \Z_3, \Z_2$ \\ \rule[-.08in]{0in}{.1in}$D_3 \times D_3, D_3 \times \Z_3, \Z_3 \times \Z_3, (\Z_3 \times \Z_3) \semi \Z_2$} & $D_4, \Z_4, (D_3 \times D_3) \semi \Z_2, (\Z_3 \times \Z_3) \semi \Z_4$ \\ \hline
\rule[-.07in]{0in}{.2in} $K_{3,3,1}$ & $D_3$, $D_2$, $\mathbb{Z}_3$ and $\mathbb{Z}_2$ & $D_4$ and $\Z_4$ \\ \hline
\rule[-.07in]{0in}{.2in}$K_{4,4}^-$ & $D_6$, $D_3$, $D_2$, $\mathbb{Z}_6$, $\Z_3$ and $\mathbb{Z}_2$ & $D_4$ and $\Z_4$ \\ \hline
\rule[-.07in]{0in}{.2in}$P_7$ & $D_3$, $\Z_3$ and $\Z_2$ & $\Z_2 \times \Z_2$ \\ \hline
\rule[-.07in]{0in}{.2in}$P_8$ & $\Z_2$ & None \\ \hline
\rule[-.07in]{0in}{.2in}$P_9$ & $D_6$, $D_3$, $D_2$, $\Z_6$, $\Z_3$ and $\Z_2$ & None \\ \hline
\rule[-.07in]{0in}{.2in}$P_{10}$ & $D_5, D_3, \Z_5, \Z_3$ and $\Z_2$ & $\Z_5 \semi \Z_4$ and $\Z_4$ \\ \hline
\end{tabular}
\end{center}
\caption{Realizable topological symmetry groups for the Petersen family.}
\label{Ta:TSG}
\end{table} 

One motivation for this project was to investigate how $\nabla Y$ and $Y\nabla$ moves affect the set of realizable topological symmetry groups of a graph. So far, we do not see any pattern to the changes, but the sample size of 7 graphs is still very small. It would still be interesting to examine another $\nabla Y$ family of graphs, such as the Heawood family of 20 graphs bookended by the Heawood graph and $K_7$. Topologically, this family contains the smallest intrinsically knotted graphs \cite{lklo}, and the topological symmetry groups for $K_7$ \cite{fmn3} and the Heawood graph \cite{flw} are already known.

\begin{question}
Which groups are realizable (and positively realizable) for each of the graphs in the Heawood family? Are there any patterns in how topological symmetry groups behave under $\nabla Y$ and $Y \nabla$-moves?
\end{question}

As we saw with $K_{3,3}$, $K_{3,3,1}$ and $K_{4,4}^-$, it may be more useful to look at subgraphs, since the symmetries of the larger graph are often restricted by the symmetries of its subgraphs.  From this perspective, it would be useful to have a census of the realizable topological symmetry groups for all ``small'' 3-connected graphs -- for example, for all such graphs with 10 or fewer vertices.

\begin{question}
Which groups are realizable (and positively realizable) for each ``small'' 3-connected graph (e.g. 10 or fewer vertices)?
\end{question}

Finally, there are many infinite families of graphs whose topological symmetry groups would be interesting to explore. Previous papers have studied complete graphs \cite{fmn3}, complete bipartite graphs \cite{hmp} and M\"{o}bius ladders \cite{fl}. Another interesting family is the {\em generalized Petersen graphs}. The generalized Petersen graph $P(n, k)$ has vertices $\{u_0, u_1, \dots, u_{n-1}, v_0, v_1, \dots v_{n-1}\}$ and edges $\{u_i, u_{i+1}\}$, $\{u_i, v_i\}$ and $\{v_i, v_{i+k}\}$ for each $i$ (with the subscripts computed modulo $n$). The Petersen graph itself is $P(5,2)$.

\begin{question}
Which groups are realizable (and positively realizable) for the generalized Petersen graph $P(n,k)$?
\end{question}

These are just a few of the many possibilities for future research into topological symmetry groups.

%%%%%%%%%%%%%%%%%%%%%%%%%%%%%%%%%%%%%%%%%%

%

% The following MDPI journals use author-date citation: Arts, Econometrics, Economies, Genealogy, Humanities, IJFS, JRFM, Laws, Religions, Risks, Social Sciences. For those journals, please follow the formatting guidelines on http://www.mdpi.com/authors/references
% To cite two works by the same author: \citeauthor{ref-journal-1a} (\citeyear{ref-journal-1a}, \citeyear{ref-journal-1b}). This produces: Whittaker (1967, 1975)
% To cite two works by the same author with specific pages: \citeauthor{ref-journal-3a} (\citeyear{ref-journal-3a}, p. 328; \citeyear{ref-journal-3b}, p.475). This produces: Wong (1999, p. 328; 2000, p. 475)

%%%%%%%%%%%%%%%%%%%%%%%%%%%%%%%%%%%%%%%%%%

%%%%%%%%%%%%%%%%%%%%%%%%%%%%%%%%%%%%%%%%%%
\end{document}